\documentclass[a4paper,11pt]{amsart}
\usepackage{caption}
\usepackage{subcaption}
\usepackage[english]{babel,varioref}
\usepackage{epsfig}
\usepackage{graphicx}
\usepackage{a4wide}
\usepackage{url}
\usepackage{srcltx}
\usepackage{color}
\usepackage{hyperref}
\newtheorem{df}{Definition}[section]

\newtheorem{thm}[df]{Theorem}
\newtheorem*{ass}{Assumption}
\newtheorem{lem}[df]{Lemma}

\begin{document}
\title[Edge based Schwarz methods]{Edge based Schwarz methods for the Crouzeix-Raviart finite volume element discretization of elliptic problems}

\author{Atle Loneland}\address{Department of Informatics, University of Bergen, 5020 Bergen, Norway}
% \thanks{NRF Grant}
\author{Leszek Marcinkowski}\address{Faculty of Mathematics, University of Warsaw, Banacha 2, 02-097 Warszawa, Poland.} \thanks{This work was partially supported by Polish Scientific Grant 2011/01/B/ST1/01179.} 
\author{Talal Rahman}\address{Department of Computing, Mathematics and Physics, Bergen University College, Norway}
% \institute {sfasdf}
%opening
\begin{abstract}
 In this paper, we present two variants of the Additive Schwarz Method for a Crouzeix-Raviart finite volume element (CRFVE) discretization of second order elliptic problems with discontinuous coefficients where the discontinuities are only across subdomain boundaries. One preconditioner is symmetric while the other is nonsymmetric. The proposed methods are almost optimal, in the sense that the residual error estimates for the GMRES iteration in the both cases depend only polylogarithmically on the mesh parameters.
%  \keywords{Finite Volume method, Crouzeix-Raviart Finite Element, Domain Decomposition, Additive Schwarz method}
\end{abstract}
\maketitle
\section{Introduction}
In this paper, we introduce two variants of the Additive Schwarz Method (ASM) for a Crouzeix-Raviart finite volume element (CRFVE) discretization of a second order elliptic problem with discontinuous coefficients, where the discontinuities are only across subdomain boundaries. Problems of this type play a crucial part in the field of scientific computation, for example, simulation of fluid flow in porous media are often affected by discontinuities in the permeability of the porous media. Discontinuities or jumps in the coefficient causes the performance of standard iterative methods to deteriorate as the discontinuities or the jumps increases. The resulting system, which in general is nonsymmetric, is solved using the preconditioned GMRES method, where in one variant of the ASM the preconditioner is symmetric while in the other variant it is nonsymmetric. The proposed methods are almost optimal, in the sense that the residual error estimates for the GMRES iteration, in the both cases, depend only polylogarithmically on the mesh parameters. 

The finite volume method divides the domain into control volumes where the nodes from the finite difference or finite element is located in the centroid of the control volume. Unlike the finite difference and the finite element method, the solution to the finite volume method satisfies conservation of certain quantities such as mass, momentum, energy and species. This property is exactly satisfied for every control volume in the domain and also for the whole computational domain. An attractive feature of this method is that it is directly connected to the physics of the system. There are two types of finite volume methods: One which is based on finite difference discretization, called the finite volume method  and one that is based on finite element discretization named the finite volume element (FVE) method. In the later the approximation of the solution is sought in a finite element space and can therefor be considered as an Petrov-Galerkin finite element method.

In the CRFVE method which is the discretization method we consider in this paper, the equations are discretized on a mesh dual to a primal mesh where the nonconforming Crouzeix-Raviart finite element space is defined, i.e., the space in which we seek the approximation of the solution, cf. \cite{Chatzipantelidis:1999:CVM}.
%, for an overview of FV methods we refer to   \cite{Lin:2013:FVEM}.

There are many results concerning Additive Schwarz Methods (ASM) for solving the symmetric system arising from finite element discretization of a model elliptic second order problems, cf. e.g. 
\cite{Toselli:2005:DDM}, but  only a few papers consider the FVE discretization based on the standard finite element space, cf. \cite{Chou:2003:ADD,Zhang:2006:ODD,Loneland:2013:AVS}.
There is  also a number of results focused  on iterative methods for the CR finite element for second order problems; cf. \cite{Brenner:1996:TLS,Marcinkowski:2008:NNA,Rahman:2005:ASM, Sarkis:1997:NCS}.

The purpose of this paper is to construct  two  parallel algorithms based on edge based discrete space decomposition in the ASM abstract scheme. This type of decomposition is the same as the one considered in \cite{Marcinkowski:2005:ADD} for a mortar type of discretization. Both methods are based on the same decomposition of the discrete space but the first one is symmetric while the second one is nonsymmetric. The algorithms are equivalent to apply parallel  ASM type of preconditioners to our CRFVE discrete problems.

We present almost optimal error bounds for the estimate of the convergence rate of GMRES method applied to our preconditioned problems, showing that the constants in the estimates  grows like $C(1+\log(H/h))^2$, where $H$ is the maximal diameter of the subdomains and $h$ is the fine mesh size parameter.
  
  For notational convenient we introduce the following notation: For positive constants $c$ and $C$ independent of $h$ we define $u\asymp v$, $x\succeq y$ and $w\preceq z$ as 
 $$cu\leq v\leq Cu,\qquad x\geq cy\qquad\text{and }w\leq Cz,\text{ respectively.}$$
 $u,v,x,y,w$ and $z$ are here norms of some functions.
\section{Prelimenaries}
\subsection{The Model Problem}
We consider the following elliptic boundary value problem
% bounded simply connected $\Omega\in\mathcal{R}$
\begin{eqnarray}
\label{eq:modelproblem}
-\nabla\cdot(A(x)\nabla u)&=&f \hspace{15 mm} \mathrm{ in }\; \Omega,\\ \nonumber
u&=&0 \hspace{15 mm} \mathrm{on}\; \partial\Omega.
\end{eqnarray}
Where $\Omega$ is a bounded convex domain in $\mathbb{R}^2$ and $f\in L^2(\Omega)$. 

The corresponding standard variational (weak) formulation is: Find $u^* \in H^1_0(\Omega)$  such that
$$
 a(u^*,v)=\int_\Omega f v \:dx \quad \forall  v \in H^1_0(\Omega),
$$
where 
$$
 a(u,v)= \sum_{k=1}^N  \int_{\Omega_k} 
\nabla u^T A(x)\nabla v
              \: dx.
$$

Now, we partition $\Omega$ into a nonoverlapping subdomains consisting of open, connected Lipschitz polytopes $\Omega_i$ such that 
$\overline{\Omega}=\bigcup_{i=1}^N \overline{\Omega}_i\,.$ We also assume that these subdomains form a coarse triangulation of the domain which is shape regular as in  \cite{Brenner:1999:BDD} with $H=\max_k H_k$, where $H_k=\text{diam }\Omega_k$.

We assume that the restriction of the symmetric coefficient  matrix to $\Omega_k$: $A_k=A_{|\Omega_k}$   is in $W^{1,\infty}(\Omega_k)$ 
%with the following bound $  \|A_k\|_{W^{1,\infty}(\Omega_k)}\leq C$ 
and  bounded and positive definite, i.e. 
\begin{eqnarray}
  \exists \alpha_k>0 \; \forall x\in \Omega_k \;\forall \xi\in {\mathbb R}^2 
\quad \xi^TA(x)\xi&\geq& \alpha_k |\xi|^2\\
 \exists M_k>0 \; \forall x\in \Omega_k \; \forall \xi,\mu\in {\mathbb R}^2 \quad \mu^TA(x)\xi&\leq& M_k |\nu| |\xi|.
\end{eqnarray}
Here $|\xi|=\sqrt{\xi^T\xi}$.
We can always scale the matrix functions  $A$ in such a way that all $\alpha_k\geq 1$.
Thus we assume that the restriction of the coefficient  matrix to $\Omega_k$: $A_k=A_{|\Omega_k}$   is in $W^{1,\infty}(\Omega_k)$ with the following bounds: 
$  \|A_k\|_{W^{1,\infty}(\Omega_k)}\leq C$, and 
  $ M_k \leq C \alpha_k$, i.e. we assume that the coefficient matrix locally is smooth, isotropic  and not too much varying.
\subsection{Basic notation}
Throughout this paper we will use the following notation for Sobolev spaces. The space of functions that have generalized derivatives of order s in the space $L^{2}(\Omega)$ is denoted as $H^s(\Omega)$. The norm on the space $H^s(\Omega)$ is defined by 
$$\|u\|_{H^s(\Omega)}=\left(\int_\Omega\sum_{|\alpha|\leq s}|D^\alpha u|^2\,dx\right)^{1/2}.$$
The space of functions with bounded weak derivatives of order $s$ is denoted by $W^{s,\infty}(\Omega)$ with the corresponding norm defined as 
$$\|u\|_{W^{s,\infty}(\Omega)}=\max_{0\leq |\alpha|\leq s}\|D^\alpha u\|_{L^2(\Omega)}.$$
The subspace of $H^1(\Omega)$, with functions vanishing on the boundary $\partial\Omega$ in the sense of traces, is denoted by $H^1_0(\Omega)$. For the duality pairing between $H^{-1}(\Omega)$ and $H^1_0(\Omega)$, we denote by $(f,u)$ the action of a functional $f\in H^{-1}(\Omega)$ on a function $u\in H^1_0(\Omega)$.

We introduce a global interface $\Gamma=\bigcup_i \overline{
  \partial \Omega_i \setminus \partial \Omega}$ which plays an important role in our study.

We assume that there exists a sequence of  quasiuniform triangulations: $\mathcal{T}_h=\mathcal{T}_h(\Omega)=\{\tau\}$, of $\Omega$ such that  any element $\tau$ of $\mathcal{T}_h$ is contained in only one subdomain, as a consequence any subdomain $\Omega_k$ inherits  a sequence of local triangulations:
$\mathcal{T}_h(\Omega_k)=\{\tau\}_{\tau\subset \Omega_k,\tau\in \mathcal{T}_h}$.
With this triangulation $\mathcal{T}_h(\Omega)$ we define the broken $H^1(\Omega)$ norm and seminorm as

\begin{eqnarray*}
 \|v\|_{H_h^1(\Omega)}=\left(\sum_{\tau\in\mathcal{T}_h(\Omega)}\|v\|^2_{H^1(\tau)}\right)^{1/2},\quad |v|_{H_h^1(\Omega)}=\left(\sum_{\tau\in\mathcal{T}_h(\Omega)}|v|^2_{H^1(\tau)}\right)^{1/2}.
\end{eqnarray*}

\begin{figure}[htb]
\centering
\includegraphics[width=0.4\textwidth,height=0.17\textheight]{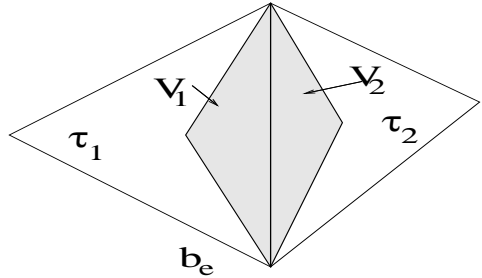}
 \caption{The control volume $b_e$ for an edge $e$ which is the common edge to the triangles $\tau_1$ and $\tau_2$.}  \label{fig:2}
\end{figure}
Let $h=\max_{\tau \in \mathcal{T}_h(\Omega)} \mathrm{diam}(\tau)$ be the mesh size parameter of the triangulation. 
We introduce the following sets of  Crouzeix-Raviart (CR) nodal points or nodes: let  $\Omega_h^{CR}, \partial \Omega_h^{CR},
\Omega_{k,h}^{CR}, \partial \Omega_{k,h}^{CR}$, $\Gamma_h^{CR}$, and $\Gamma_{kl,h}^{CR}$
be the midpoints of edges of elements in $\mathcal{T}_h$ which are on  $\Omega, \partial\Omega, \Omega_k, \partial\Omega_k$, $\Gamma$, and $\Gamma_{k l}$, respectively.
Here $\Gamma_{kl}$ is an interface, an open edge, which is shared by the two subdomains, $\Omega_k$ and $\Omega_l$. Note that 
$
 \Gamma_h^{CR}=\bigcup_{\Gamma_{k l}\subset \Gamma} \Gamma_{kl,h}^{CR}.
$
%  Fig. CR element
%*****************************
\begin{figure}[htb]
\centering
\includegraphics[width=0.2\textwidth]{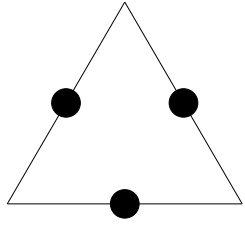}
 \caption{The degrees of freedom of the nonconforming Crouzeix-Raviart element.}  \label{fig:1}
\end{figure}
Now we define a dual triangulation $\mathcal{T}_h^*$ to the initial one.
For an  edge $e$ of an element not on $\partial\Omega$, i.e., a common edge $e$ for two elements $\tau_1$ and $\tau_2$, $e$ is defined as $e=\partial \tau_1 \cap \partial \tau_2$. We now introduce two triangles: $V_k\subset \tau_k$ obtained  by connecting the ends of $e$ to the centroid (barycenter) of $\tau_k$ for $k=1,2$. Then, let the  control volume   $b_e=V_1\cup e \cup V_2$, cf.~Figure~\ref{fig:2}.   For an edge of an element $\tau$ contained in $\partial \Omega$ let the control volume be the triangle $V$ obtained analogously i.e. by connecting the ends of $e$ with the centroid of $\tau$.
Then let $\mathcal{T}_h^*=\{b_e\}_{e\in E_h}$, where $E_h$ is the set of all edges of elements in $\mathcal{T}_h$.
\subsection{Discrete problem}
In this section we present the Crouzeix-Raviart finite element (CRFE) and finite volume (CRFV) discretizations of a model second order elliptic problem with discontinuous coefficients across prescribed substructures boundaries. We define the two discrete spaces mentioned above as:
\begin{eqnarray*}
 V_h &:=& \{u\in L^2(\Omega): 
                v_{|\tau}\in P_1, \quad  \tau \in \mathcal{T}_h \quad %piecewise linear 
                v(m)=0\quad  m \in \partial\Omega_h^{CR} %discrete bnd condition
               \}, \\
 V_h^* &:=& \{u\in L^2(\Omega): 
                v_{|b_e}\in P_0,   \quad  b_e \in \mathcal{T}_h^* \quad %piecewise constant
                v(m)=0\quad m \in \partial\Omega_h^{CR} %discrete bnd condition
               \} .
\end{eqnarray*}
The first space is the classical nonconforming Crouzeix-Raviart finite element space, cf.~Figure~\ref{fig:1}, and the second space is the space of piecewise constant functions which are zero on the boundary of the domain. Both spaces are contained in $L^2(\Omega)$.

Let $\{\phi_m\}_{m\in \Omega_h^{CR}}$ be the standard CR nodal basis of $V^h$ and 
 $\{\psi_m\}_{m\in \Omega_h^{CR}}$ be the standard basis of $V_h^*$ 
consisting of characteristic functions of the control volumes.
% i.e. 
%\begin{eqnarray*}
%  \phi_m(s)=\psi_m(s)=\left\{
% \begin{array}{lr} 1 &\quad m=s\\ 0  & m\not=s  \end{array}
%\right.  \quad m,s \in \Omega_h^{CR}.
%\end{eqnarray*}

We also introduce two interpolation operators, $I_h$ and $I_h^*$, defined for any function that has properly defined and unique values at each midpoint $m \in  \Omega_h^{CR}$:
\begin{eqnarray*}
 I_h(u)=\sum_{m \in \Omega_h^{CR}} u(m) \phi_m, \qquad 
  I_h^*(u)=\sum_{m \in \Omega_h^{CR}} u(m) \psi_m.
\end{eqnarray*}

Note that $I_hI_h^*u=u$ for any $u\in V_h$ and $I_h^*I_h u=u$ for any $u\in V_h^*$.
Now we define a nonsymmetric in  general bilinear form $a_h:V_h\times V_h^*\rightarrow {\mathbb R}$:
\begin{eqnarray}
  a_h^{CRFV}(u,v)=- \sum_{e\in E_h^{in}} v(m_e)\int_{\partial b_e}A(s)\nabla u\cdot\mathbf{n}  \: ds ,
\end{eqnarray}
where $\textbf{n}$ is a normal unit vector  outer to $\partial b_e$, $m_e$ is the median (midpoint)  of the edge $e$ and  $E_h^{in}\subset E_h$ is the set of all interior  edges, i.e. those which are not on $\partial \Omega$.

Then our discrete CRFV problem  is to find $u_h^{FV} \in V_h$ such that:
%\begin{equation}  a_h^{CRFV}(u,v) =f(v) \quad \forall v\in V_h^* \end{equation}
%or equivalently 
\begin{equation}
\label{eq:disc_pr}
  a_h^{FV}(u_h^{FV},v)=f(I_h^* v) \qquad \forall v \in V_h
\end{equation}
for $a_h^{FV}(u,v):=a_h^{CRFV}(u,I_h^*v)$.
In general this problem is nonsymmetric unless the coefficients matrix is a piecewise constant matrix over each element  $\tau\in \mathcal{T}_h(\Omega)$. 
One can prove that there exists $h_0>0$ such that for all $h\leq h_0$  the form $a_h^{FV}(u,v)$ is positive definite over $V_h$. Thus  this problem has a unique solution. Some  error estimates are also proven, cf. \cite{Loneland:2013:AVS} or \cite{Chatzipantelidis:1999:CVM} in the case of the smooth coefficients. 

The corresponding symmetric nonconforming finite element problem is defined as: Find $u^{FE}\in V_h$ such that: 
\begin{equation}
\label{eq:disc_pr2}
 a_h(u_h^{FE},v)=\left(f,v\right),\qquad v\in V_h.
\end{equation} 
The bilinear form $a(\cdot,\cdot)$ also induces the so called energy norm which is defined as $\|\cdot\|_a=\sqrt{a(\cdot,\cdot)}$.

The next lemma is crucial for the analysis of our method. It relates the CRFV and CRFE bilinear forms. The proof for the type of problems under consideration in this paper can be found in \cite{Loneland:2013:AVS}.
\begin{lem} For the bilinear forms $a^{FE}(u,v)$ and $a^{FV}(u,v)$ there exists $h_0>0$ such that the following holds
\label{lem:conv}
 \begin{eqnarray}
  |a_h^{FE}(u,v)-a_h^{FV}(u,I_h^*v)|\preceq h\|u\|_a\|v\|_a,\qquad \forall u,v\in V_h.
 \end{eqnarray}
%  and 
%  \begin{equation}
%  \label{eq:ellipticity}
%   a_h(u,I_h^*u)\geq C\|u\|_{1,h}^2
%  \end{equation}
% where C is a positive constant independent of $h$.

\end{lem}
\section{The GMRES Method}
The linear system of equations which arises from problem (\ref{eq:disc_pr}) is in general nonsymmetric. We may solve such a system using a preconditioned GMRES method; cf. Saad and Schultz \cite{saad1986gmres} and Eistenstat, Elman and Schultz \cite{eisenstat1983variational}. This method has proven to be quite powerful for a large class of nonsymmetric problems. The theory originally developed for $L^2(\Omega)$ in \cite{eisenstat1983variational} can easily be extended to an arbitrary Hilbert space; see \cite{cai1989some,Cai:1992:DDA}.

In this paper, we use GMRES to solve the linear system of equations
\begin{equation}
 Tu=g,
\end{equation}
where $T$ is a nonsymmetric, nonsingular operator, $g\in V_h$ is the right hand side and $u\in V_h$ is the solution vector.

The main idea of the GMRES method is to solve a least square problem in each iteration, i.e. at step $m$ we approximate the exact solution $u^*=T^{-1}g$ by a vector $u_m\in \mathcal{K}_m$ which minimizes the norm of the residual, where $\mathcal{K}_m$ is the  $m$-th Krylov subspace defined as 
$$\mathcal{K}_m=\text{span}\left\{r_0,Tr_0,\cdots T^{m-1}r_0\right\}$$ and $r_0=g-Tu_0$.
In other words, $z_m$ solves 
\begin{equation*}
 \min_{z\in\mathcal{K}_m}\|g-T(u_0+z)\|_a.
\end{equation*}
Thus, the $m$-th iterate is $u_m=u_0+z_m$.

The convergence rate of the GMRES method is usually expressed in terms of the following two parameters
\begin{equation*}
 c_p=\inf_{u\neq0}\frac{a(Tu,u)}{\|u\|_a^2}\text{ and }C_p=\sup_{u\neq0}\frac{\|Tu\|_a}{\|u\|_a}.
\end{equation*}
The decrease of the norm of the residual in a single step is described in the next theorem.
\begin{thm}[Eisenstat-Elman,Schultz]
If $c_p>0$, then the GMRES method converges and after m steps, the norm of the residual is bounded by
\begin{equation}
 \|r_m\|_a\leq\left(1-\frac{c_p^2}{C_p^2}\right)^{m/2}\|r_0\|_a,
\end{equation}
where $r_m=g-Tu_m$.
\end{thm}
The two parameters describing the convergence rate of the GMRES method will be estimated in Theorem \ref{thm:mainthmfve} once the proposed domain decomposition preconditioner corresponding to the operator $T$ is defined and analyzed.

\section{Additive Schwarz Method}
In this section we introduce the additive method for the discrete problem (\ref{eq:disc_pr}) and provide bounds on the convergence rate, both for the solution of the symmetric and nonsymmetric problem following the newly developed abstract framework of \cite{marcinkowskirahman2013}. For each substructure $\Omega_k$ define the restriction of $V^h$ to $\bar\Omega_k$ and the corresponding subspace with CR zero Dirichlet boundary conditions as $$W_k:=\left\{v_{\bar\Omega_k}: v\in V_h\right\}$$ and $$W_{k,0}:=\left\{v\in W_h:v(m)=0 for m\in\partial\Omega_{k,h}^{CR}\right\},$$
respectively. Clearly $W_k\subset W_{k,0}$.
Now let $P_k: W_k\rightarrow W_{k,0}$ be the orthogonal projection of a function $u\in V^h$ onto $W_{k,0}$ defined by 
\begin{equation}
 a^{FE}_{k,h}(P_ku,v)=a^{FE}_{k,h}(u,v)\qquad \forall v\in W_{k,0},
\end{equation}
and define $H_ku=u-P_ku$ as the discrete harmonic counterpart of $u$, i.e.
\begin{eqnarray}
\label{eq:discreteharmonic}
  a_{k,h}^{FE}(H_k u,v)&=&0 \qquad  \forall v \in W_{k,0},\\
  H_ku(m)&=&u(m) \qquad m \in \partial\Omega_{k,h}^{CR}.
\end{eqnarray}
A function $u\in W_k$ is locally discrete harmonic if $H_ku=u$. If all restrictions to subdomains of a function $u\in V^h$ are locally discrete harmonics, i.e.,
$$u_{|\Omega_k}=H_k u_{|\Omega_k}\qquad \text{for } k=1,\ldots,N$$ then we say $u$ is a discrete harmonic function. 

For any function $u\in V^h$, this gives a decomposition of $u$ into locally discrete harmonic parts and local projections, i.e. $u=Hu+Pu$ where $Hu=(H_1u,\ldots,H_Nu)$ and $Pu=(P_1u,\ldots,P_Nu)$.

An important property of discrete harmonic functions is the minimal energy one. A discrete harmonic function $u=H_ku$ has minimal energy among all functions which are equal to $u$ on $\partial\Omega^{CR}_{k,h}$, i.e.
\begin{equation}
 a_k(u,u)=\min\left\{a_k(v,v):v(p)=u(p)\quad\forall p\in\partial\Omega^{CR}_{k,h}\right\}.
\end{equation}
Another important property is that the values of a discrete harmonic functions in the interior CR nodal points of subdomains are completely determined by the values on $\partial\Omega_{k,h}^{CR}$ and (\ref{eq:discreteharmonic}).

\subsection{Decomposition of $V_h(\Omega)$}
To define our additive Schwarz method  we first need to define a decomposition of the space $V_h(\Omega)$ into subspaces equipped with local bilinear forms. 

We start by defining special edge functions which we will use to build our coarse space.
\begin{df}
 \label{def:edgefunc}
 Let $\Gamma_{kl}\subset\Gamma$ be and edge and let $\theta_{kl}\in V^h$ be a discrete harmonic function defined at the CR nodal points on $\Gamma_{kl}$ as follows
 \begin{itemize}
  \item $\theta_{kl}(p)=1$ for $p\in\Gamma^{CR}_{kl,h}$,
  \item $\theta_{kl}(p)=0$ for $p\in\Gamma^{CR}_h\setminus\Gamma^{CR}_{kl,h}$.
 \end{itemize}
\end{df}
The coarse space is then defined as the span of these edge functions, i.e., $V_0=\mathrm{span}\{\theta_{kl}\}\subset V_h(\Omega)$. The support of an edge function $\theta_{kl}$ corresponding to an interface $\Gamma_{kl}$, i.e., an edge shared by the two subdomains $\Omega_k$ and $\Omega_l$, is contained in $\Omega_{k}\cup\Omega_{l}\cup\Gamma_{kl}$, cf. Figure~\ref{fig:edgefuncsupport}.
\begin{figure}[htb]
        \centering
         \includegraphics[width=\linewidth]{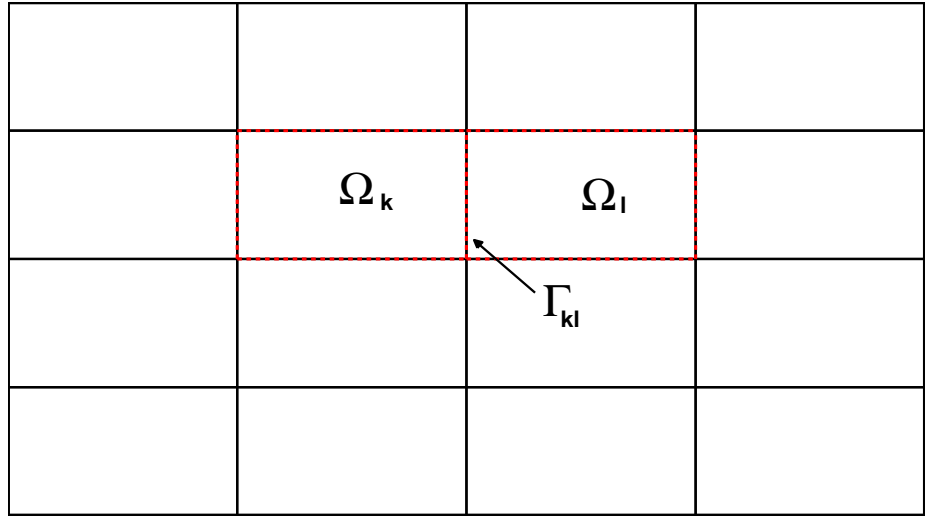}
%  \rule{4cm}{3cm}
  \caption{Support of an edge function $\theta_{kl}$ corresponding to the interface $\Gamma_{kl}$.}
\label{fig:edgefuncsupport}

\end{figure}

The local spaces corresponding to $\Gamma_{kl}$ are defined as the space of discrete harmonic functions which are nonzero only at nodal points in $\Omega_{k,h}^{CR}\cup\Omega_{l,h}^{CR}\cup\Gamma_{kl,h}^{CR}$. We define the bilinear form for these local spaces to be the original bilinear form restricted to $\Omega_k$, i.e., $a_{kl}(u,v)=a_{k,h}(u,v)$. The last set of subspaces for our decomposition is the one corresponding to the subregions $\Omega_k$. Let $V_k$ be the space $W_{k,0}$ extended by zero to all remaining subdomain. This yields the following decomposition of our discrete space $V_h(\Omega)$: $$V_h=V_0+\sum_{\Gamma_{kl}\subset\Gamma}V_{kl}+\sum_{k=1}^N V_k.$$

Now we define the symmetric and nonsymmetric projection like operators:

For $i=0,\cdots,N$ the projection operators $T_i^{sym}\colon V_h\rightarrow V_i$ for the coarse and local subdomains are defined as
\begin{equation*}
 a_h^{FE}(T^{sym}_iu,v)=a_h^{FV}(u,v)
 \qquad\forall v\in V_i(\Omega).
\end{equation*}
The projection operator $T_{kl}^{sym}\colon V_h\rightarrow V_{kl}$ associated with the edge $\Gamma_{kl}$ is defined as
\begin{equation*}
 a_h^{FE}(T^{sym}_{kl}u,v)=a_h^{FV}(u,v)
 \qquad\forall v\in V_{kl}.
\end{equation*}
Note that $T^{sym}_{kl}$ is defined as the extension with zeros to all remaining subdomains of the local projection operator $P_ku_{|\Omega_k}$ and may be computed by solving local symmetric  discrete CRFE Dirichlet problem.

The nonsymmetric operator which is based solely on the nonsymmetric bilinear form $a_h^{FV}(u,v)$ is defined completely analogously: 

For $i=0,\cdots,N$ the projection operators $T_i^{nsym}\colon V_h\rightarrow V_i$ for the coarse and local subdomains are defined as
\begin{equation*}
 a_h^{FV}(T^{nsym}_iu,v)=a_h^{FV}(u,v)
 \qquad\forall v\in V_i(\Omega).
\end{equation*}
Similarly as in the symmetric case, the edge related operator $T_{kl}^{sym}\colon V_h\rightarrow V_{kl}$ associated with the edge $\Gamma_{kl}$ is defined as
\begin{equation*}
 a_h^{FV}(T^{nsym}_{kl}u,v)=a_h^{FV}(u,v)
 \qquad\forall v\in V_{kl}.
\end{equation*}
Each of these problems have a unique solution. We now introduce 
\begin{eqnarray*}
 T^{type}:=\sum_{\Gamma_{kl}\subset\Gamma}T_{k l}^{type}+\sum_{k=0}^N T_k^{type},
\end{eqnarray*}
where the super-index $type$ is either $sym$ or $nsym$ corresponding to the symmetric and nonsymmetric operators.
This allow us to replace the original problem (\ref{eq:disc_pr}) by the equation 
\begin{eqnarray}\label{eq:precond}
  T^{type} u^{FV}_h = g^{type}.
\end{eqnarray}
where $g^{type}$ is defined as
$$g^{type}=g_0^{type} +\sum_{\Gamma_{k l} \subset \Gamma} g_{k l}^{type} + \sum_{k=1}^N g_k^{type}$$
with $g_0=T_0^{type}u^{FV}_h$,  $g_{k l}^{type}=T_{k l}^{type}u^{FV}_h $ and $g_k^{type}=T_k^{type}u^{FV}_h$ for ${type} \in\{sym,nsym\}$ .
Note that $g^{type}_i$ may be computed without knowing the solution $u^{FV}_h$ of (\ref{eq:disc_pr}).
\subsection{Analysis}
\label{sect:analysis}
Before we state the main theorem regarding the convergence rate of our proposed method we state two auxiliary lemmas without proofs which will help us analyze and estimate the parameters describing GMRES convergence rate. The proofs may be found in \cite{Marcinkowski:2005:ADD} and references therein.
\begin{lem}
\label{lem:coarsebasisbound}
 Let $\Gamma_{kl}\subset\Gamma $ be and edge and let $\theta_{kl}$ be an edge function from Definition \ref{def:edgefunc}. Then for any $u\in V_h(\Omega_i)$ we have
 \begin{eqnarray}
 \label{eq:coarsebasisbound}
  |\theta_{kl}|^2_{H^1_h(\Omega_i)}&\preceq&\left(1+\log\left(\frac{H_i}{h_i}\right)\right),\\  
%   |u_{kl}|^2_{H^1_h(\Omega_i)}&\asymp&\|\mathcal{M}_iu_{kl}\|^2_{H^{1/2}_{00}(\Gamma_{kl})},\\
  |u_{kl}|^2_{H^1_h(\Omega_i)}&\preceq&\left(1+\log\left(\frac{H_i}{h_i}\right)\right)^2(H_i^{-2}\|u\|^2_{L^2(\Omega_i)}+|u|^2_{H_h^1(\Omega_i)}),\nonumber
 \end{eqnarray}
where $u_{kl}$ is a function taking the same values as $\theta_{kl}u$ at the CR nodal points on $\partial\Omega_i$.
\end{lem}

\begin{lem}
 \label{lem:estimate}
 For any $u\in V_0$ the following holds
 \begin{equation}
  a(u,u)\preceq \sum_{k=1}^NM_k\left(1+\log\left(\frac{H_i}{h_i}\right)\right)\left(\sum_{\Gamma_{kl}\neq\Gamma_{ik}}(u_{kl}-u_{\Gamma_{kj}})(u_{kl}-u_{\Gamma_{kj}})\right),\\
 \end{equation}
where the second sum is taken over all pairs of edges $\Gamma_{kl},\Gamma_{ik}\subset\partial\Omega_k$.
\end{lem}

We are now ready to state the main theorem for the convergence rate of our ASM applied to nonsymmetric problem (\ref{eq:disc_pr}).
\begin{thm}
\label{thm:mainthmfve}
 There exists $h_0>0$ such that for all $h<h_0$, $k=1,2,$ and $u\in V_h$, we have
\begin{eqnarray*}
 \|T^{type}u\|_a&\preceq&\|u\|_a,  \\
a(T^{type}u,u)&\succeq& \left(1+\log\left(\frac{H_i}{h_i}\right)\right)^{-2}
 \: a(u,u) ,
\end{eqnarray*}
\end{thm}
\begin{proof}
Following the framework of \cite{marcinkowskirahman2013} we need to prove three key assumptions.
\begin{ass}[1] There exists $h_0>0$ such that for all $u,v\in V_h$ the following holds
  \begin{eqnarray}
  |a_h^{FE}(u,v)-a_h^{FV}(u,I_h^*v)|\preceq h\|u\|_a\|v\|_a,
%   \qquad \forall u,v\in V_h.
 \end{eqnarray}
\end{ass}
This is just Lemma~\ref{lem:conv}.
\begin{ass}[2]
\label{ass:sym1}
For all $u\in V^h$ there exists a constant $C>0$ such that there is a representation $u=u_0+\sum_{i=1}^N u_i+\sum_{kl}u_{kl},\; $ with $u_0\in V_0,u_i\in V_i,u_{kl}\in V_{kl}$, such that
\begin{equation*}
 a(u_0,u_0)+\sum_{i=1}^Na(u_i,u_i)+\sum_{kl}a(u_{kl},u_{kl})\leq C\leq\left(1+\log\left(\frac{H_i}{h_i}\right)\right)^2a(u,u).
\end{equation*}

\end{ass}
This assumption is the same as Assumption 1 in the standard Schwarz framework for domain decomposition methods, cf (\cite{smith1996domain, Toselli:2005:DDM}). To verify the assumption we first need to define a decomposition of the function $u\in V^h$. Following the lines of the proof of Lemma 6.1 in \cite{Marcinkowski:2005:ADD} we start by letting $u_0\in V_0$ be defined by $u_0=\sum_{kl}\bar{u}_{kl}\theta_{kl}$, where $\bar{u}_{kl}$ is an average of $u$ over $\Gamma_{kl}$.

Next, let $w=u-u_0$ and define $u_k=P_kw$ for each subspace $V_k$. Note that $P_kw=P_ku$ since $u_0$ is discrete harmonic and also $w-\sum_{k=1}^Nu_k$ is discrete harmonic in each subdomain. The decomposition for $V_{kl}$ is straightforward. For an edge $\Gamma_{kl}$ define $u_{kl}\in V_{kl}$ at the CR nodes of $\Gamma_{kl}$ as
$$u_{kl}(p)=\theta_{kl}(p)w(p),\qquad\forall p\in \Gamma_{kl,h}^{CR}.$$ Above we have used the fact that $u_k\in V_k$ are equal to zero in $\bigcup_{k=1}^N\partial\Omega_{k,h}^{CR}$, i.e., $u_k$ are equal to zero at all CR nodes on the boundary of any substructures. Clearly this yields $u=u_0+\sum_{i=1}^N u_i+\sum_{kl}u_{kl}$.

To validate the estimate of Assumption \ref{ass:sym1} we start by estimating $a(u_0,u_0)$. From Lemma~\ref{lem:estimate} and Schwarz inequality we have
\begin{eqnarray*}
 a(u_0,u_0)&\preceq&\sum_{k=1}^NM_k\left(1+\log\left(\frac{H_i}{h_i}\right)\right)\sum_{\Gamma_{kl},\Gamma_{ik}\subset\partial\Omega_k}|\bar u_{kl}-\bar u_{\Gamma_{kj}}|^2\\
 &\preceq&\sum_{k=1}^NM_k\left(1+\log\left(\frac{H_i}{h_i}\right)\right)\frac{1}{H}\sum_{\Gamma_{kj}\subset\partial\Omega_k}\|u-\bar u_{kl}\|^2_{L^2(\Gamma_{kj})},
\end{eqnarray*}
where $\Gamma_{kl}$ is an arbitrary edge of $\Omega_k$. Applying standard trace theorem arguments and Poincare's inequality for nonconforming elements, cf. \cite{Sarkis:1997:NCS,Brenner:1996:TLS}, we get 
\begin{equation}
\label{eq:coareestimate}
 a(u_0,u_0)\preceq\left(1+\log\left(\frac{H}{\underline h}\right)\right)a(u,u).
\end{equation}
This takes care of the term corresponding to the coarse space. Next, we consider the the term $u_k\in V_k$ associated with the interior subspaces. Using the fact that $P_k$ is an orthogonal projection with respect to the local bilinear form $a_k(\cdot,\cdot)$ and Lemma~\ref{lem:estimate} we have
\begin{eqnarray*}
\sum_{k=1}^Na(u_k,u_k)&=&\sum_{K=1}^Na_k(u_k,u_k)=\sum_{k=1}^Na_k(P_kw,P_kw)\leq a(w,w),\\
&\preceq&a(u_0,u_0)+a(u,u).
\end{eqnarray*}
From (\ref{eq:coareestimate}) we then get
\begin{equation}
\label{eq:localestimate}
 \sum_{k=1}^N a(u_k,u_k)\preceq\left(1+\log\left(\frac{H}{\underline h}\right)\right)a(u,u),
\end{equation}
which completes the estimate for the local components.

Next, we need to bound the term associated with the edge subspaces. By (\ref{eq:coarsebasisbound}) in Lemma \ref{lem:coarsebasisbound} and Poincare's inequality for nonconforming elements we get
\begin{eqnarray*}
 a_k(u_{kl},u_{kl})&\leq&M_k|u_{kl}|^2_{H^1_h(\Omega_k)},\\
 &\preceq&M_k\left(1+\log\left(\frac{H_i}{h_i}\right)\right)^2\left(\frac{1}{H_i^2}\|u-\bar u_{kl}\|^2_{L^2(\Omega_k)}+|u-\bar u_{kl}|^2_{H^1_h(\Omega_k)}\right)\\
 &\preceq&M_k\left(1+\log\left(\frac{H_i}{h_i}\right)\right)^2|u|^2_{H^1_h(\Omega_k)}\preceq\frac{M_k}{\alpha_k}\left(1+\log\left(\frac{H_i}{h_i}\right)\right)^2a_k(u,u)\\
 &\preceq&\left(1+\log\left(\frac{H_i}{h_i}\right)\right)^2a_k(u,u)
\end{eqnarray*}
Summing the above estimate over all edges we get
\begin{eqnarray}
\label{eq:edgeestimate}
 \sum_{kl}a_k(u_{kl},u_{kl})&\preceq&\sum_{k=1}^N\left(1+\log\left(\frac{H_i}{h_i}\right)\right)^2a_k(u,u),\nonumber\\
 &\leq&\left(1+\log\left(\frac{H}{\underline h}\right)\right)^2a(u,u),
\end{eqnarray}
Summing (\ref{eq:coareestimate}),(\ref{eq:localestimate}) and (\ref{eq:edgeestimate}) completes the proof.
% \begin{ass}[3]
% Let $0\leq\mathcal{E}_{ij}\leq1$ to be the minimal values that satisfy 
% \begin{equation*}
%  a(u_i,u_j)\leq\mathcal{E}_{ij}a(u_i,u_i)^{1/2}a(u_j,u_j)^{1/2},\qquad\forall u_i\in V_i,\;\forall u_j\in V_j.
% \end{equation*}
% with $i,j=1,\ldots,N$ or $i,j=\Gamma_{kl}\subset\Gamma$. Define $\rho(\mathcal{E})$ to be the spectral radius of $\mathcal{E}=\{\mathcal{E}_{ij}\}$.
% \end{ass}

The last assumption we need to prove is the one involving Strengthened Cauchy-Schwarz inequalities. This is the same assumption given in the standard Schwarz framework for the convergence theory of domain decomposition methods, cf. \cite{smith1996domain, Toselli:2005:DDM}.
The spectral radius of the constants from these inequalities may be bounded using a standard coloring argument.

This completes the proof.
\end{proof}
% \clearpage
\section{Numerical results}
In this section, we present some numerical results for the proposed method. 
All experiments are done for problem~\ref{eq:modelproblem} on a unit square domain $\Omega=(0,1)^2$ for the symmetric preconditioner, i.e., for $type=sym$. The coefficient $A$ is equal to $2+\sin(100\pi x)\sin(100\pi y)$, except for regions (subdomains) marked with red where $A$ equals $\alpha_1(2+\sin(100\pi x)\sin(100\pi y))$, where $\alpha_1$ is a parameter describing the jump in the coefficient (cf. Figure~\ref{fig:3} and Table~\ref{tbl:it}). The right hand side is chosen as $f=1$.
The numerical solution is found by using the generalized minimal residual method (GMRES). 
\noindent
\begin{figure}[htb]
\centering
\begin{subfigure}[b]{0.48\linewidth}
        \centering
         \includegraphics[trim=145 60 145 60, clip,width=\linewidth]{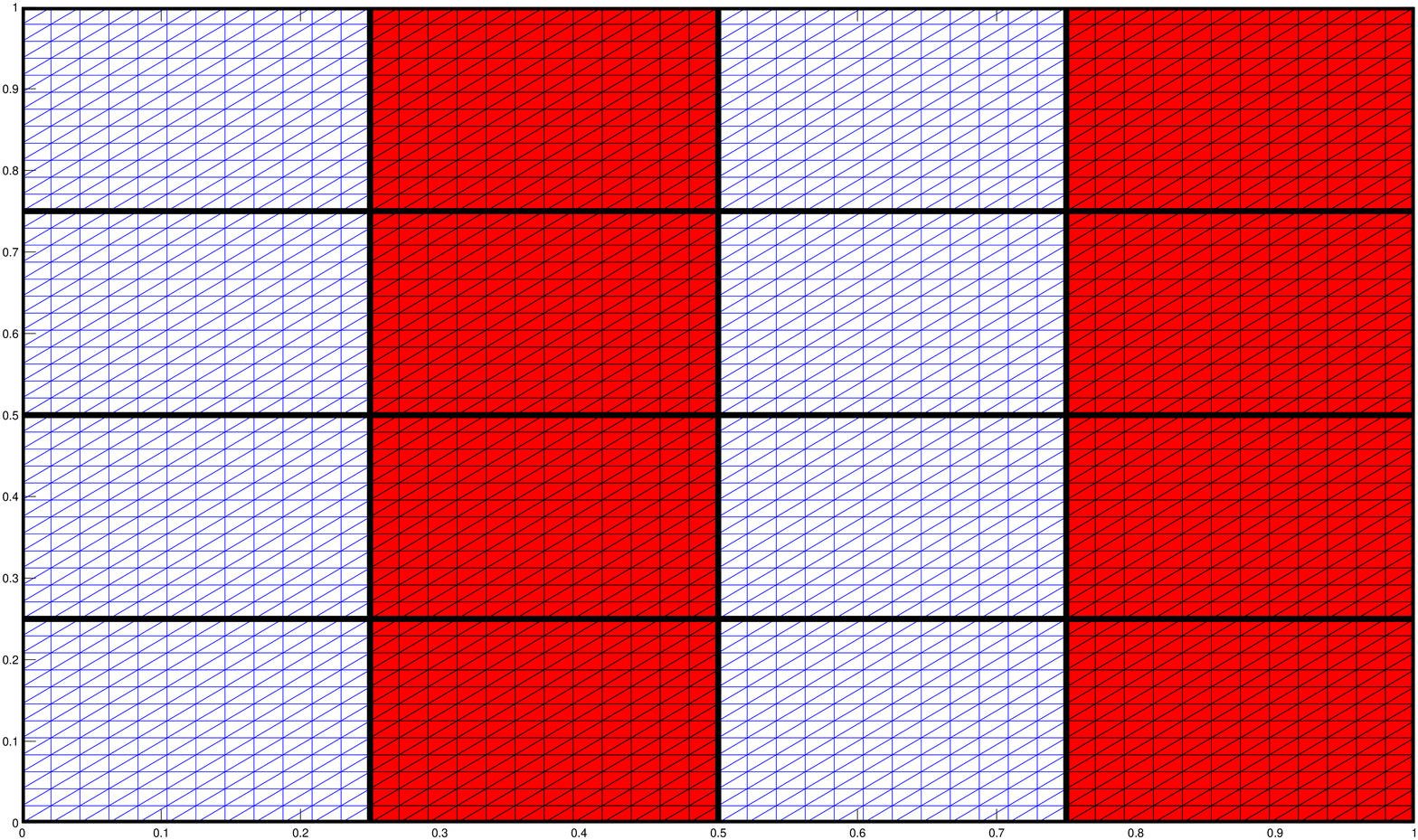}
%  \rule{4cm}{3cm}
  \caption{Problem 1.}
% \label{fig:1}

\end{subfigure}
\hspace{0.1cm}
\begin{subfigure}[b]{0.48\linewidth}
        \centering
%  \rule{4cm}{3cm}
 \includegraphics[trim=145 60 145 60, clip,width=\linewidth]{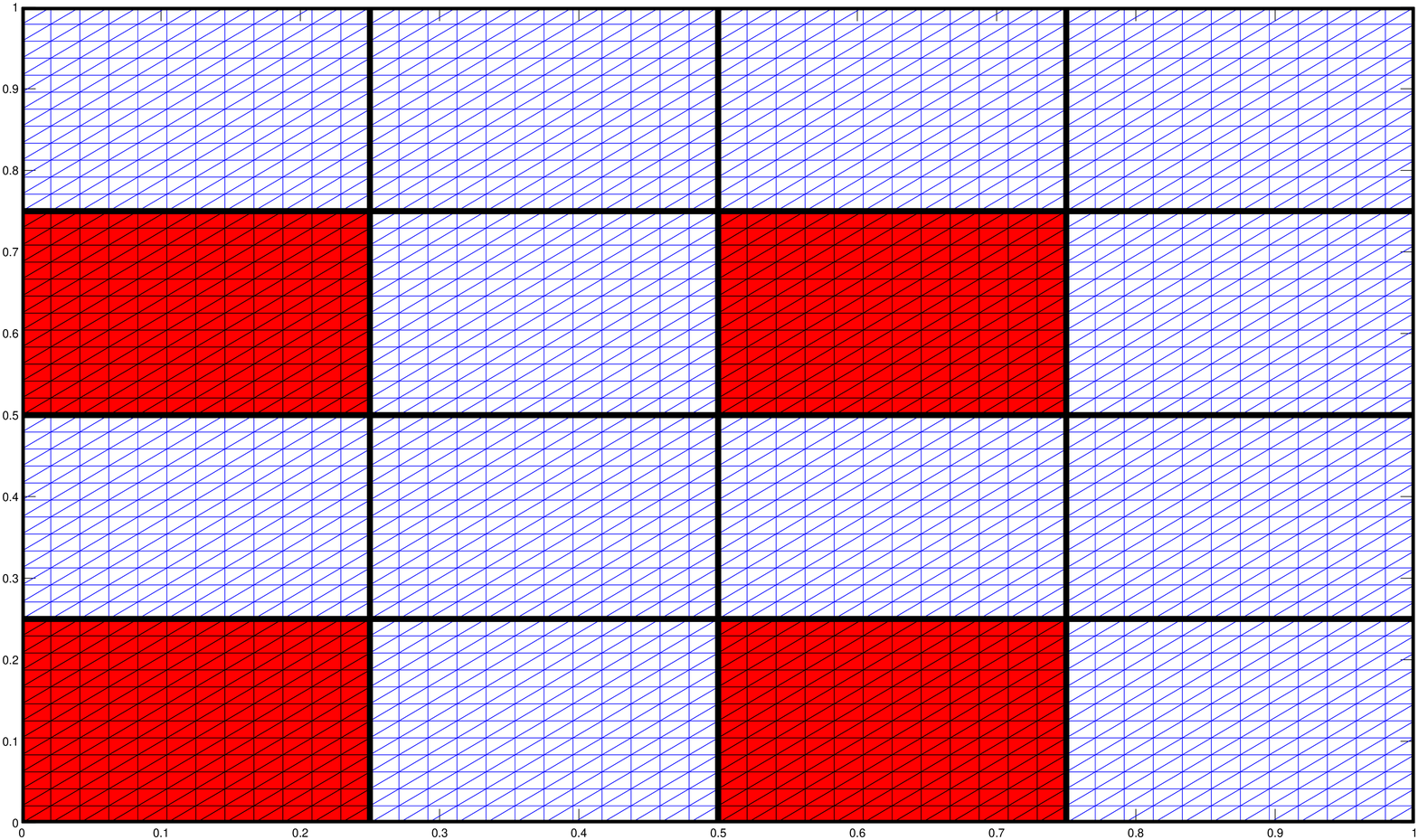}
 \caption{Problem 2. }
% \label{fig:2}
\end{subfigure}
\caption{Test problems 1 and 2. Regions (subdomains) marked with red are where $A$ depends on $\alpha_1$. Fine mesh consists of $48\times48$ rectangular blocks, while coarse mesh consists of $4\times4$ rectangular subdomains}
\label{fig:3}
\end{figure}
\begin{figure}[htb]
        \centering
         \includegraphics[width=\linewidth]{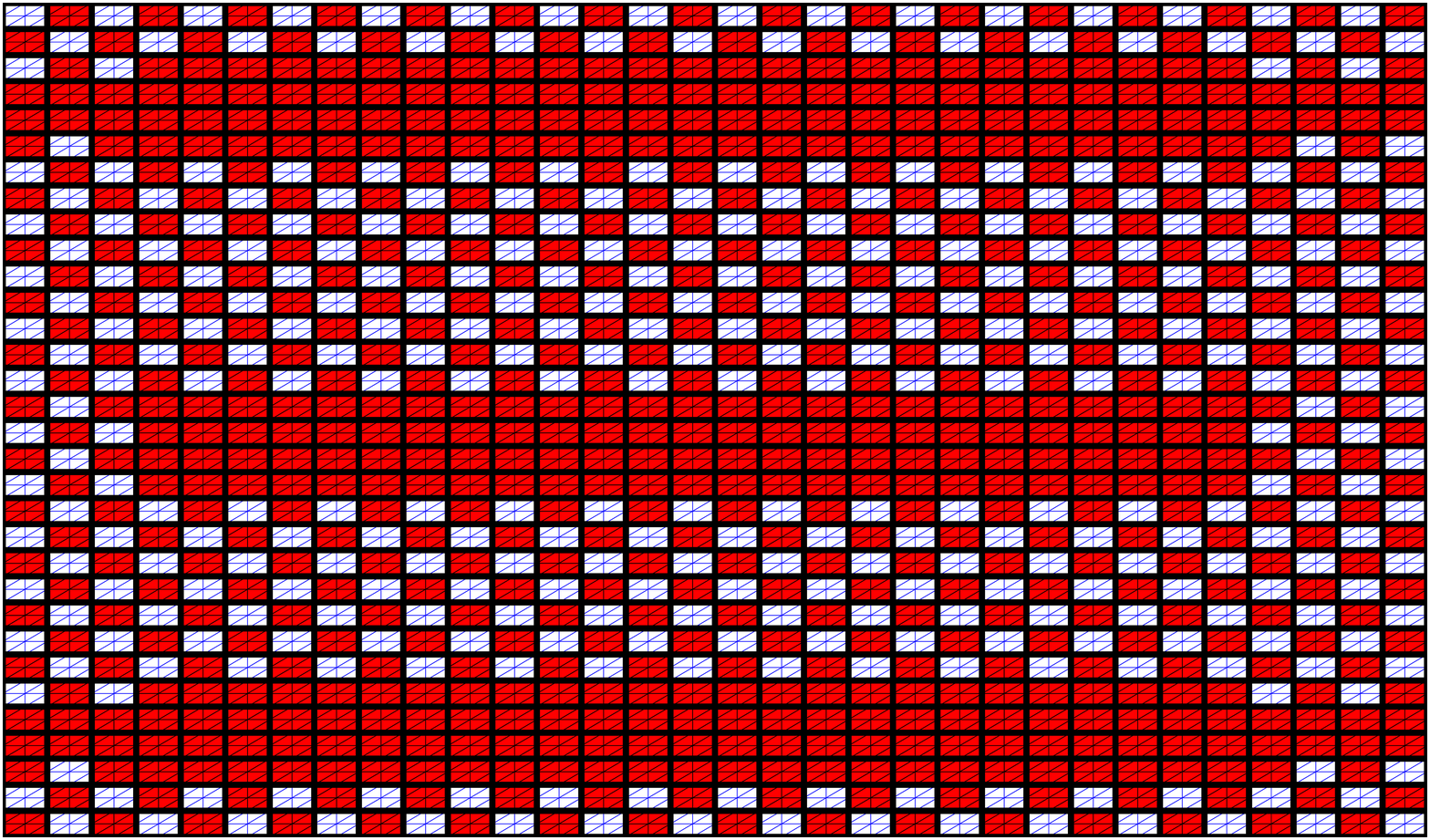}
%  \rule{4cm}{3cm}
  \caption{Test problem 3. Regions (subdomains) marked with red are where $A$ depends on $\alpha_1$. Fine mesh consists of $64\times64$ rectangular blocks, while coarse mesh consists of $32\times32$ rectangular subdomains}
\label{fig:4}

\end{figure}

\noindent
\begin{figure}[htb]
\centering
\begin{subfigure}{0.48\linewidth}
        \centering
         \includegraphics[trim=60 10 100 20, clip,width=\linewidth]{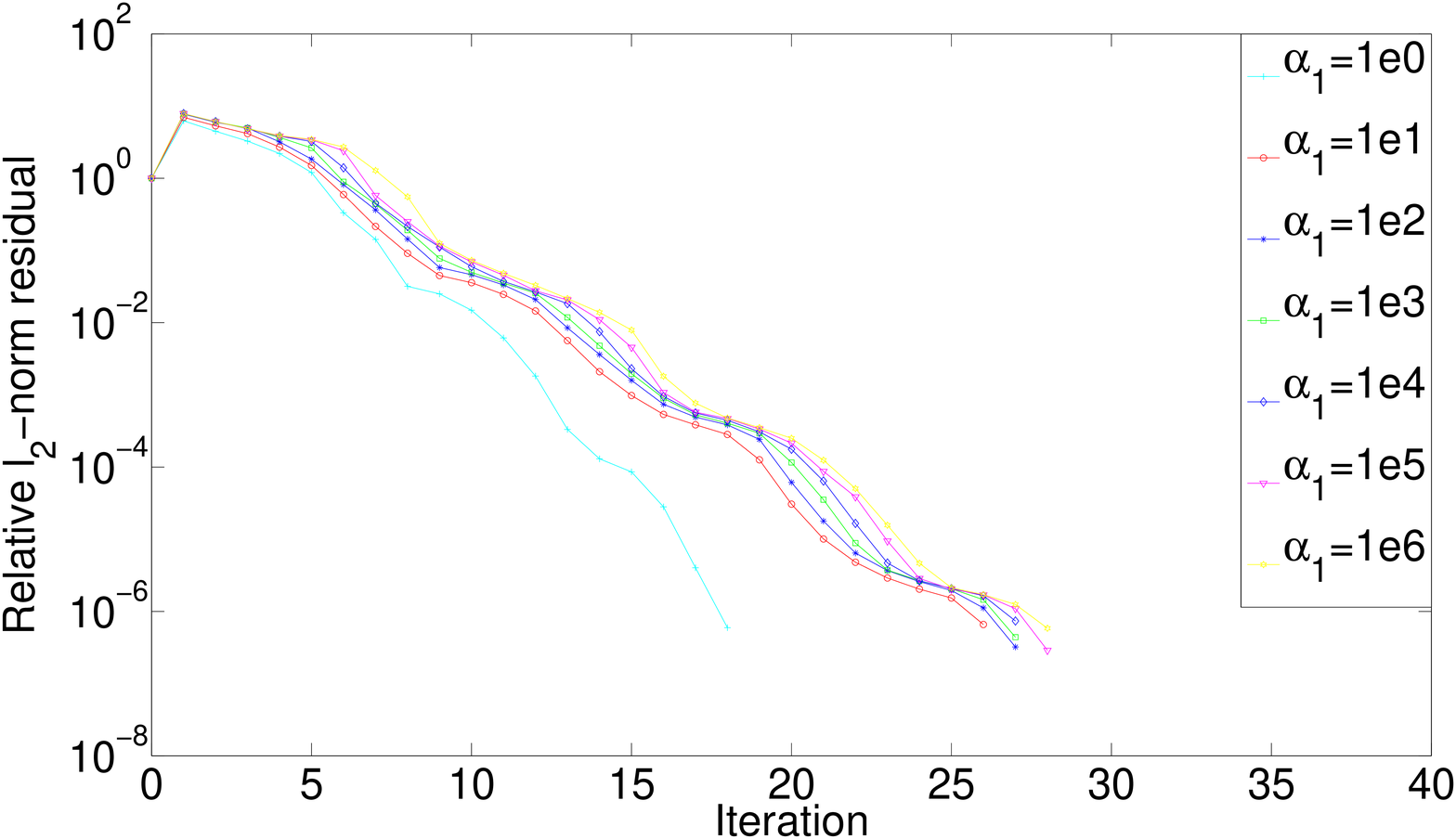}
%  \rule{4cm}{3cm}
  \caption{Problem 1.}
 \label{fig:prob1resplot}

\end{subfigure}
\hspace{0.1cm}
\begin{subfigure}{0.48\linewidth}
        \centering
%  \rule{4cm}{3cm}
 \includegraphics[trim=60 10 100 20, clip,width=\linewidth]{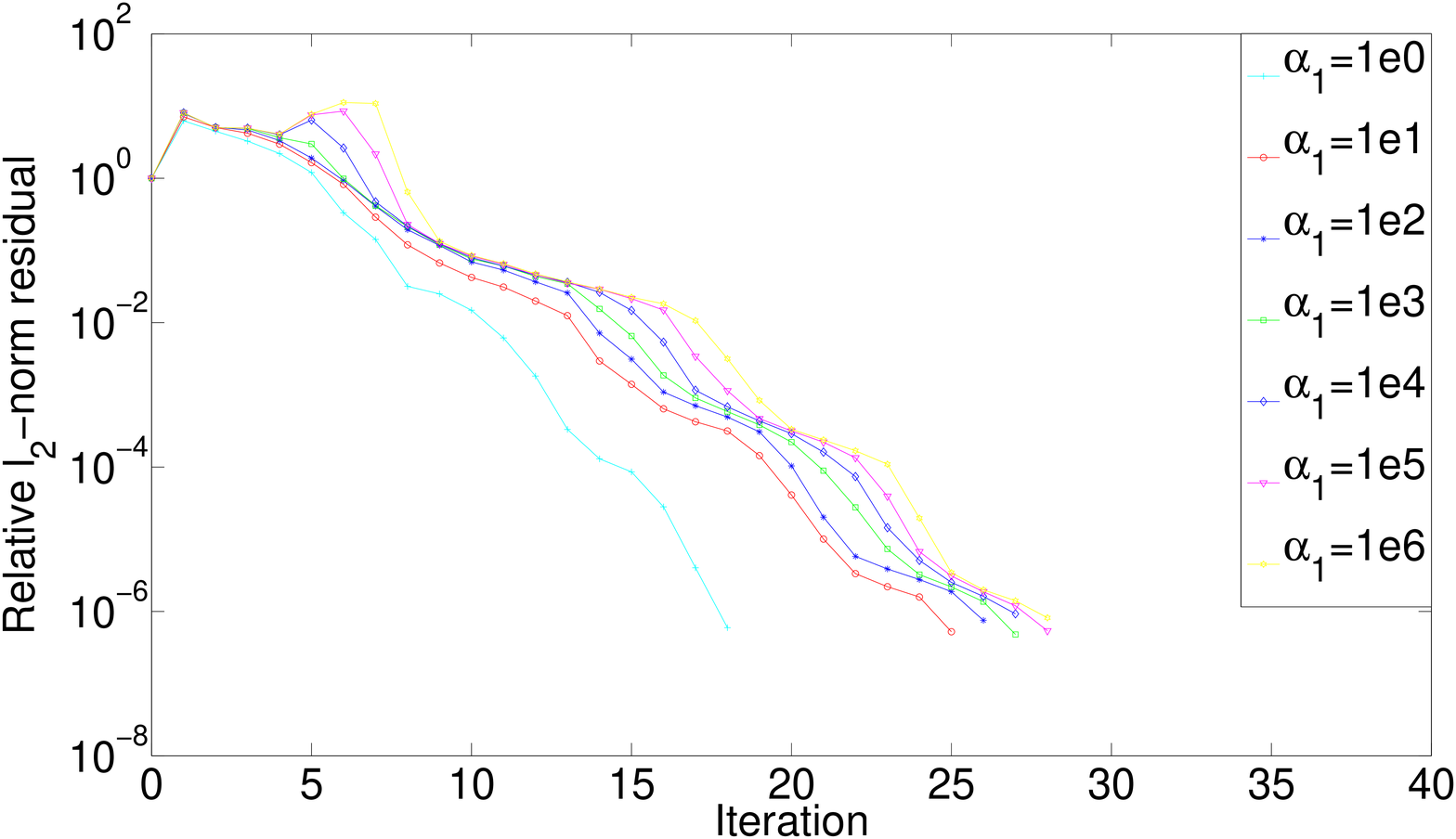}
 \caption{Problem 2.}
 \label{fig:prob2resplot}
\end{subfigure}
\begin{subfigure}{0.48\linewidth}
        \centering
%  \rule{4cm}{3cm}
 \includegraphics[trim=60 10 100 20, clip,width=\linewidth]{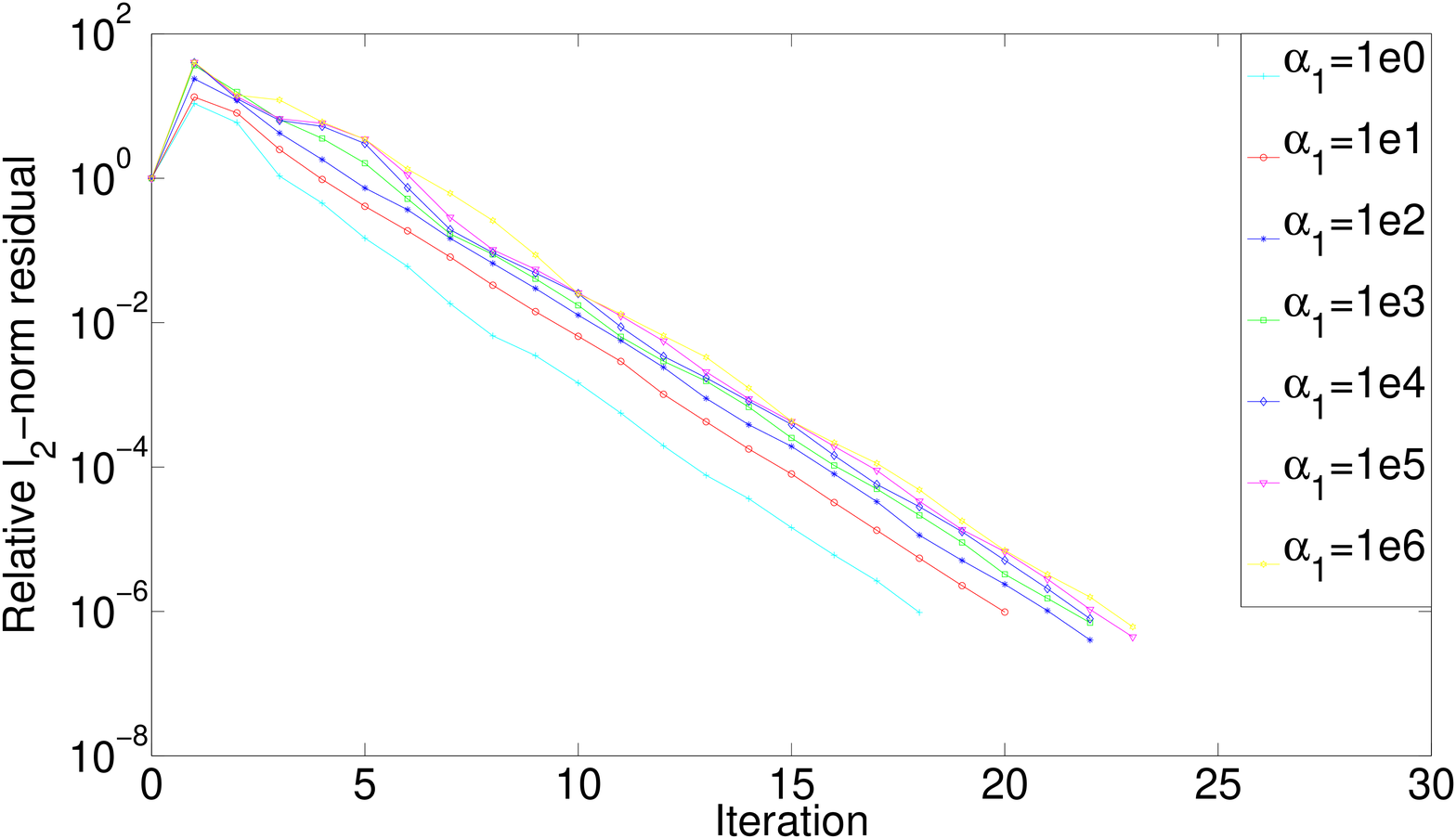}
 \caption{Problem 3.}
 \label{fig:prob3resplot}
\end{subfigure}
 \caption{Relative residual norms for GMRES minimizing the $A$-norm for different values of $\alpha_1$.}
\label{fig:relresplot}
\end{figure}
We run the method until the $l_2$ norm of the residual is reduced by a factor of $10^{6}$, that is when $\|r_i\|_2/\|r_0\|_2\leq 10^{-6}$. The number of iterations and estimates of the smallest eigenvalue for the different types of problems under consideration, are shown in Table~\ref{tbl:it}--\ref{tbl:asstab2}. 

We first consider the two test problems with discontinuities over subdomain boundaries as shown in Figure~\ref{fig:3} for a fine mesh $h=1/32$ and coarse mesh $H=1/4$. In Problem 3 we extend the two previous problems into a larger and more complicated problem with respect to the distribution and discontinuities of the coefficient $A$, see Figure~\ref{fig:4}. The fine and coarse mesh parameters are here $h=1/128$ and $H=1/32$, respectively. The number of iterations used by the preconditioned GMRES method are reported in Table~\ref{tbl:it} with the smallest eigenvalue of the preconditioned system (\ref{eq:precond}) shown in the parentheses next to the iteration numbers. In Figure~\ref{fig:prob1resplot}--\ref{fig:prob3resplot} we have plotted the relative residuals for these problems measured in the $l_2$ norm.

In Table \ref{tbl:asstab1} and \ref{tbl:asstab2} we show the asymptotic dependency on the mesh parameters $H$ and $h$ for two test cases where the coefficient $A$ is equal to $2+\sin(100\pi x)\sin(100\pi y)$ and $2+\sin(10\pi x)\sin(10\pi y)$, respectively.

\begin{table}
\centering
\begin{tabular}{l l l l}
% \hline
% &Average ASM&\\
% \hline
&Problem 1:&Problem 2:&Problem 3:\\
\hline
$\alpha_1$ &$\sharp$ iter.&$\sharp$ iter.&$\sharp$ iter.\\
\hline
%   1     &17(11.41) \\
%   $10^0$&17(6.82)&21(12.20)& 1321\\
  $10^0$&18(2.15e-1)&18(2.15e-1)&18(4.73e-1)\\
  $10^1$&25(2.14e-1)&26(2.09e-1)&20(4.89e-1)\\
  $10^2$&26(2.14e-1)&27(2.07e-1)&22(4.88e-1)\\
  $10^3$&27(2.14e-1)&27(2.06e-1)&22(4.84e-1)\\
  $10^4$&27(2.14e-1)&27(2.06e-1)&22(4.78e-1)\\
  $10^5$&27(2.14e-1)&28(2.06e-1)&23(4.77e-1)\\
  $10^6$&28(2.14e-1)&28(2.06e-1)&23(4.77e-1)\\
\hline
\end{tabular}
\vspace{5mm}
\caption{Number of GMRES iterations until convergence for the solution of (\ref{eq:disc_pr}), with different values of $\alpha_1$ describing the coefficient $A$ in the red regions, cf. figures \ref{fig:3} and \ref{fig:4}.} 
\label{tbl:it}
\end{table}
\begin{table}
\centering
\begin{tabular}{|l| l l l l l l |}
\hline
% &Average ASM&&&\\
% \hline
% $\alpha_1$&Problem 1:&Problem 2:&Problem 3:&Problem 4:\\
% \hline
$h/H$ &$1/4$&$1/8$&$1/16$&$1/32$&$1/64$&$1/128$\\
\hline
%   1     &17(11.41) \\
%   $10^0$&17(6.82)&21(12.20)& 1321\\
  $1/8$&13(5.31e-1)&&&&&\\
  $1/16$&16(3.47e-1)&17(4.86e-1)&&&&\\
  $1/32$&17(2.23e-1)&20(3.44e-1)&17(4.85e-1)&&&\\
  $1/64$&19(1.62e-1)&24(2.51e-1)&20(3.46e-1)&17(4.85e-1)&&\\
  $1/128$&21(1.24e-1)&28(1.86e-1)&24(2.60e-1)&20(3.45e-1)&16(4.85e-1)&\\
  $1/256$&24(9.84e-2)&32(1.41e-1)&29(1.90e-1)&23(2.63e-1)&19(3.47e-1)&16(4.85e-1)\\
\hline
\end{tabular}
\vspace{5mm}
\caption{Iteration number for increasing values of $h$ and $H$ with $A=2+\sin(10\pi x)\sin(10\pi y)$.} 
\label{tbl:asstab1}
\end{table}
\begin{table}
\centering
\begin{tabular}{|l| l l l l l l |}
\hline
% &Average ASM&&&\\
% \hline
% $\alpha_1$&Problem 1:&Problem 2:&Problem 3:&Problem 4:\\
% \hline
$h/H$ &$1/4$&$1/8$&$1/16$&$1/32$&$1/64$&$1/128$\\
\hline
%   1     &17(11.41) \\
%   $10^0$&17(6.82)&21(12.20)& 1321\\
  $1/8$&12(5.32e-1)&&&&&\\
  $1/16$&14(3.64e-1)&17(4.85e-1)&&&&\\
  $1/32$&16(2.64e-1)&19(3.45e-1)&18(4.73e-1)&&&\\
  $1/64$&19(1.87e-1)&22(2.60e-1)&21(3.36e-1)&18(4.73e-1)&&\\
  $1/128$&22(1.39e-1)&28(1.82e-1)&25(2.52e-1)&22(3.37e-1)&20(4.65e-1)&\\
  $1/256$&24(1.07e-1)&35(1.26e-1)&34(1.66e-1)&25(2.56e-1)&25(3.26e-1)&19(4.78e-1)\\
\hline
\end{tabular}
\vspace{5mm}
\caption{Iteration number for increasing values of $h$ and $H$ with $A=2+\sin(100\pi x)\sin(100\pi y)$.} 
\label{tbl:asstab2}
\end{table}

The iteration numbers and eigenvalue estimates in Table \ref{tbl:it} reflects well the theoretical results developed in Section \ref{sect:analysis}. We see no dependency on the contrast in $A$ when the jumps in the coefficient are over subdomains, see Figure~\ref{fig:3}--\ref{fig:4}. The iteration numbers and the eigenvalue estimates in Table~\ref{tbl:asstab1}--\ref{tbl:asstab2} confirms our theory that the parameters describing the convergence of the GMRES method only depends polylogarthmically on the mesh ratio $\frac{H}{h}$.

\clearpage
 \bibliography{CRFVEEdgeBased-LMR}
% \addcontentsline{toc}{chapter}{Bibliography}
 \bibliographystyle{plain}
\end{document}